\documentclass[12pt]{amsart}
\usepackage{amsmath}
\usepackage{amssymb}
\usepackage{amsfonts}
\usepackage{amsthm}
\usepackage{xcolor}
\usepackage{enumerate}
\usepackage{hyperref}
\hypersetup{
    bookmarks=true,         
    colorlinks=true,       
    linkcolor=blue,          
    citecolor=red!50!black,        
    filecolor=red!20!black,         
    urlcolor=magenta        
}

\newtheorem{theo}{Theorem}[section]
\newtheorem*{theo*}{Theorem}
\newtheorem{coro}[theo]{Corollary}

\newtheorem{rema}[theo]{Remark}

\newcommand{\N}{\mathbb{N}}

\newcommand{\cqs}{c_q^{(s)}}

\title[]{On an identity of Delange and its application to Cohen-Ramanujan expansions}
\begin{document}

 \author{Vinod Sivadasan}
 \address{Department of Mathematics, University College, Thiruvananthapuram (Research Centre of the University of Kerala), Kerala - 695034, India}
 \address{Department of Mathematics, College of Engineering Trivandrum, Thiruvananthapuram, Kerala - 695034, India}
 \address{Department of Collegiate Education, Government of Kerala, India}
 \email{wenods@gmail.com}

\author[K V Namboothiri]{K Vishnu Namboothiri}
\address{Department of Mathematics, Baby John Memorial Government College, Chavara, Sankaramangalam, Kollam, Kerala - 691583, India}
\address{Department of Collegiate Education, Government of Kerala, India}
\email{kvnamboothiri@gmail.com}
\thanks{}

 \begin{abstract}
Srinivasa Ramanujan provided Fourier series expansions of certain arithmetical functions in terms of the exponential sum defined by $c_q(n)=\sum\limits_{\substack{{m=1}\\(m,q)=1}}^{q}e^{\frac{2 \pi imn}{q}}$.   Later, H. Delange derived the bound $\sum\limits_{q|k}|c_q(n)|\leq n\, 2^{\omega(k)}$ and gave a sufficient condition for such expansions to exist.  A. Grytczuk gave an exact value for this bound, and derived a converse implication of the absolute convergence stated by H. Delange. We here show that these results have natural generalizations in terms of the Cohen-Ramanujan sum $c_q^{(s)}(n)$ defined by E. Cohen in [\emph{Duke Mathematical Journal, 16(85-90):2, 1949}]. We derive a bound as well as exact value for $\sum\limits_{q|k}|c_q^{(s)}(n)|$ and provide a sufficient condition for the Cohen-Ramanujan expansions to exist.
 \end{abstract}
 
 \maketitle
\section{Introduction}
In \cite{ramanujan1918certain}, Srinivasa Ramanujan gave Fourier series like expansions for certain arithmetical functions. These pointwise expansions were of the form $f(n)=\sum\limits_{q=1}^{\infty}a_q c_q(n)$ where $a_q$ are complex numbers and  $c_q(n)$ denotes the \emph{Ramanujan sum} defined as
\begin{align*}
c_q(n) :=\sum\limits_{\substack{{m=1}\\(m,q)=1}}^{q}e^{\frac{2 \pi imn}{q}}.
\end{align*}
 Later, A. Wintner \cite{wintner1943eratosthenian} proved that if 
 \begin{align}\label{eq:wintner}
 \sum\limits_{k=1}^{\infty}\frac{|f'(k)|}{k}<\infty,                                                                                                                                                                                                                                                                       \end{align} 
 where                                                                                                                                                                                                                                                                        \begin{align}\label{eq:f-star-mu}
f'(k) = \sum\limits_{d|k}\mu(d)f\left(\frac{k}{d}\right)                                                                                                                                                                                                                                                                       \end{align} and $\mu$ is the M\"{o}bius function,
then with $a_q = \sum\limits_{m=1}^{\infty}\frac{f(mq)}{mq}$, $\sum\limits_{q=1}^{\infty}a_q c_q(n)$ converges absolutely to $f(n)$.  

H. Delange improved the inequality (\ref{eq:wintner}) in \cite{delange1976ramanujan}. He proved that the absolute convergence of $\sum\limits_{q=1}^{\infty}a_q c_q(n)$ holds with the weaker condition $\sum\limits_{k=1}^{\infty} 2^{\omega(k)} \frac{|f'(k)|}{k} <\infty$ where $\omega(k)$ is the number of distinct prime divisors of $k$. The most important part of his proof to establish this was the inequality $\sum\limits_{q|k}|c_q(n)|\leq n\,2^{\omega(k)}$. Later, A. Grytczuk gave an exact equality for the sum $\sum\limits_{q|k}|c_q(n)|$ in \cite{Grytczuk1981}. He used this equality to prove that the absolute convergence implication given by H. Delange has an easy converse (\cite[Theorem 1]{grytczuk1992ramanujan}).

Ramanujan sum was generalized by many later. E.\ Cohen gave a generalization of the Ramanujan sum in \cite{cohen1949extension}. He defined the sum
\begin{align}\label{cohen-ram-sum}
c_q^s(n) :=\sum\limits_{\substack{{h=1}\\(h,q^s)_s=1}}^{q^s}e^{\frac{2\pi i n h}{q^s}},
\end{align} where $(a,b)_s$ is the generalized gcd of $a$ and $b$. A. Chandran and K. V. Namboothiri derived some conditions for the existence of Ramanujan series like expansions for certain arithmetical functions in terms of the generalized Ramanujan sum (\ref{cohen-ram-sum}) (henceforth called as the  \emph{Cohen-Ramanujan sum})  in \cite{chandran2023ramanujan}. The method of arguments they used were those employed by L. Lucht in \cite{lucht1995ramanujan} and \cite{lucht2010survey}. We will call the expansions in terms of the Cohen-Ramanujan sums by the name \emph{Cohen-Ramanujan expansions}.

The aim of this paper is to show that the bounds and other results mentioned above proved by H. Delange and A. Grytczuk have natural generalizations involving Cohen-Ramanujan sums. Hence we will give a bound and an exact value for the sum $\sum\limits_{q|k}|c_q^{(s)}(n)|$ and  use them to give a sufficient condition for the Cohen-Ramanujan expansions to exist.

\section{Notations and basic results}

By $(a, b)_s$, we mean the \emph{generalized GCD} of $a$ and $b$ defined to be the largest $d^s \in \mathbb{N}$ ($d\in \N$)  such that $d^s|a$ and $d^s|b$. Hence $(a,b)_1$ is the usual GCD $(a,b)$. By $J_s$ we mean the \emph{Jordan totient function} defined as $J_s(n) := n^s \prod\limits_{\substack{{p \mid n}}}(1-\frac{1}{p^s})$. Note that $J_1(n)=\varphi(n)$ where $\varphi$ is the Euler totient function.

\emph{The Cohen-Ramanujan sum} (or \emph{generalized Ramanujan sum}) given by E. Cohen in \cite{cohen1949extension} is the sum given in equation (\ref{cohen-ram-sum}).
By (\cite[Theorem 1]{cohen1949extension}), $\cqs(n)$ is multiplicative in $q$. When $q$ is a prime power, the values of the Cohen-Ramanujan sum are as folows (\cite[Theorem 3]{cohen1949extension}).

 \begin{align}
  c_{p^j}^{(s)}(n) = \begin{cases}
                         p^{sj}-p^{s(j-1)} &\quad \text{if }p^{sj}|n\\
                         -p^{s(j-1)} &\quad \text{if }p^{s(j-1)}||n\\
                         0&\quad  \text{if }p^{s(j-1)}\nmid n.
                        \end{cases}\label{eq:crs-primes}
 \end{align}
  Cohen-Ramanujan sum satisfies the following relation with the M\"{o}bius function $\mu$ (\cite[Theorem 3]{cohen1949extension}):
\begin{align}
                    c_r^{s}(n)=\sum\limits_{\substack{d|r\\d^s|n}}\mu(r/d)d^s\label{eq:crs-mu}.
                   \end{align}
 Also, with the Jordan totient function, it holds the relation (\cite[Theorem 2]{cohen1959trigonometric})
 \begin{align}
       c_r^{(s)}(n) = \frac{J_s(n)\mu(m)}{J_s(m)}\quad\text{ where } m=\frac{n}{(r,n)}\label{eq:crs-jordan}.
      \end{align}
Sometimes we club the conditions $d|r$ and $d^s|n$ into the single condition $d^s|(n, r^s)_s$.

By the notation $d^m||n$ we mean that $d^m|n$ and $d^{m+1}\nmid n$.
For a prime $p$, the exponent of $p$ in the prime factorization of $k$  will be denoted by $e_p(k)$. Hence $p^{e_p(k)}||k$. Analogously, we write $a := e_p^{(s)}(n)$ if $p^{as}|n$ and $p^{(a+1)s}\nmid n$. In this case, we also use the notation $p^{as}||_s n$.
We use the M\"{o}bius inversion formula for the Dirichlet product of two arithmetical functions which states that 
\begin{align}
g(k) = (\mu*f)(k) :=\sum\limits_{d|k}\mu(d)f\left(\frac{k}{d}\right) \Longleftrightarrow f(k) = (g*u)(k) = \sum\limits_{d|k} g(k) 
\end{align} where $u\equiv 1$ is the unit function defined by $u(n)=1$ for all $n\in \N$ .
 \section{Main Results and proofs}

We start this section giving a generalization of the  bound $\sum\limits_{q|k}|c_q(n)|\leq n\, 2^{\omega(k)}$ given by H. Delange in \cite{delange1976ramanujan}.

 \begin{theo}\label{th:gen-delange}
  The bound $\sum\limits_{q|k}|c_q^{(s)}(n)|\leq n. 2^{\omega(k)}$ holds for all positive integers $k, n$, and $s$.
 \end{theo}
 \begin{proof}
  Since $\cqs(n)$ is multiplicative in $q$, $|\cqs(n)|$ is also multiplicative in $q$. Then $h_n(k) = (c_{\bullet}^{(s)}(n) * u)(k)= \sum\limits_{q|k}|\cqs(n)|u(k/q) = \sum\limits_{q|k}|\cqs(n)|$ is multiplicative in $k$ it being the Dirichlet product two multiplicative functions. Hence we may write $h_n(k)=h_n(\prod\limits_{p|k} p^{e_p(k)}) = \prod\limits_{p|k} h_n(p^{e_p(k)})$. Now for any positive integer $r$, we have $h_n(p^r) = \sum\limits_{q|p^r}|\cqs(n)| =  \sum\limits_{j=0}^r |c_{p^j}^{(s)}(n)|$.   
  If $p^s\nmid n$, then by property \ref{eq:crs-primes}, $c_p^{(s)}(n) = -1$ and $c_{p^j}^{(s)}(n) = 0$ for any $j>1$. By identity \ref{eq:crs-mu}, $c_1^{(s)}(n) = 1$. Therefore, if $p^s\nmid n$, then $h_n(p^r) = |c_1^{(s) }(n)| + |c_p^{(s)}(n)|=2$.
  
  Now suppose that $p^s|n$. If $p^{rs}|n$, that is, $r\leq e_p^{(s)}(n)$, then 
  \begin{align*}
   h_n(p^r) =& |c_1^{(s)}(n)|+|c_p^{(s)}(n)|+\ldots +|c_{p^r}^{(s)}(n)|\\
   =&1+|p^s-1|+|p^{2s}-p^s|+\ldots+|p^{rs}-p^{(r-1)s)}|\\
   =&p^{rs}\\
   \leq& 2p^{e_p^{(s)}(n)s}.
  \end{align*}
On the other hand, if $ e_p^{(s)}(n)=r-m$ for some positive integer $m$, then 
  \begin{align*}
   h_n(p^r) =& |c_1^{(s)}(n)|+|c_p^{(s)}(n)|+\ldots +|c_{p^{r-m}}^{(s)}(n)|+|c_{p^{r-m+1}}^{(s)}(n)|+\ldots+|c_{p^r}^{(s)}(n)|\\
   =&1+|p^s-1|+\ldots+|p^{(r-m)s}-p^{(r-m+1)s}|+|-p^{(r-m)s}|+0+0+\ldots\\
   =& 2p^{(r-m)s}\\
   =& 2p^{e_p^{(s)}(n)s}.
  \end{align*}
Therefore $h_n(p^r)=2$ if $p^s\nmid n$ and $h_n(p^r) \leq 2p^{e_p^{(s)}(n)s}$ if $p^s | n$.                       
Hence for distinct primes $p$, we get \begin{align*}      
h_n(k)  &= \prod\limits_{\substack{p|k\\ p^s|n}}h_n(p^{e_p(k)})\times \prod\limits_{\substack{p|k\\ p^s\nmid n}}h_n(p^{e_p(k)})\\
&\leq  \prod\limits_{\substack{p|k\\ p^s|n}}2p^{e_p^{(s)}(n)s}\times \prod\limits_{\substack{p|k\\ p^s\nmid n}}2\\
&\leq  \prod\limits_{\substack{p^s|n}}p^{e_p^{(s)}(n)s}\times \prod\limits_{\substack{p|k}}2\\
&\leq n\,2^{\omega(k)}
      \end{align*}
  \end{proof}

\begin{rema}\label{rem:del-equality}
For equality to hold in the above, it should be true that $p|k\Rightarrow p^s|n$ and $p^r||k \Rightarrow r >e_p^{(s)}(n)$.  In short, the bound is an equality when $n=m^s$ and $k$ is a multiple of $m\prod\limits_{p|m}p$. So the bound in the theorem is the best possible.
 \end{rema}
 
Let  $f'$ be defined by equation \ref{eq:f-star-mu}.  We now derive a sufficient condition for the Cohen-Ramanujan series expansions to exist for $f$.
\begin{theo}\label{th:suff-condition}
 Let $a_q = \sum\limits_{m=1}^{\infty}\frac{f'(mq)}{(mq)^s}$. If $\sum\limits_{k=1}^{\infty} 2^{\omega(k)} \frac{|f'(k)|}{k^s} <\infty$, then $\sum\limits_{q=1}^{\infty} a_qc_q^{(s)}(n^s)$ converges absolutely to $f(n)$.
\end{theo}

\begin{proof}
 We will first prove the absolute convergence of $\sum\limits_{q=1}^{\infty} a_qc_q^{(s)}(n^s)$ assuming the bound given. We have
 \begin{align*}
  |\sum\limits_{q=1}^{\infty} a_q c_q^{(s)}(n^s)| &= \sum\limits_{m,q\geq 1}^{\infty}\frac{|f'(mq)|}{(mq)^s} |c_q^{(s)}(n^s)|\\
  &= \sum\limits_{k=1}^{\infty}  \sum\limits_{mq=k} \frac{|f'(k)|}{k^s} |c_q^{(s)}(n^s)|\\
  &= \sum\limits_{k=1}^{\infty}  \frac{|f'(k)|}{k^s} \sum\limits_{q|k}  |c_q^{(s)}(n^s)|\\
  &\leq n^s \sum\limits_{k=1}^{\infty}  \frac{|f'(k)|}{k^s} 2^{\omega(k)}
 \end{align*}
 which is bounded by the assumption.
 Now we will show that the series converges to $f(n)$.
 \begin{align*}
  \sum\limits_{q=1}^{\infty} a_q c_q^{(s)}(n^s) &= \sum\limits_{m,q\geq 1}^{\infty}\frac{f'(mq)}{(mq)^s} c_q^{(s)}(n^s)\\
  &= \sum\limits_{k=1}^{\infty}\frac{f'(k)}{k^s} \sum\limits_{\substack{mq=k}} c_q^{(s)}(n^s).
  \end{align*}
  Fix $k$. With thee notation $e(x)=exp(2\pi i x)$, we have
  \begin{align*}
   \sum\limits_{\substack{mq=k\\m=1}}^k c_q^{(s)}(n^s) &= \sum\limits_{m|k} \, \sum\limits_{\substack{(h,k^s/m^s)_s=1\\h=1}}^{k^s/m^s} e(\frac{hn^s}{k^s/m^s})\\
   &= \sum\limits_{m|k} \, \sum\limits_{\substack{(h,k^s)_s=m^s\\h=1}}^{k^s} e(\frac{hn^s}{k^s})\\
   &= \sum\limits_{h=1}^{k^s} e(\frac{hn^s}{k^s})\\
   &= \begin{cases}
       k^s\quad\text{if } k^s|n^s\text{ or equivalently, if }k|n\\
       0\quad\text{otherwise}.
      \end{cases}
  \end{align*}
Therefore

\begin{align*}
  \sum\limits_{q=1}^{\infty} a_q c_q^{(s)}(n^s) &= \sum\limits_{\substack{k=1\\k|n}}^{\infty}\frac{f'(k)}{k^s}k^s = f(n)\text{ by the Mobius inversion formula}.
\end{align*}
\end{proof}

 For the sum discussed in Theorem \ref{th:gen-delange}, A. Grytczuk (\cite{Grytczuk1981}) gave the equality  $$\sum\limits_{q|k}|c_q(n)| = 2^{\omega(\frac{k}{(k,n)})} (k,n).$$ We will derive an analogous equality in the next theorem and use it to give a converse for the above theorem, as Grytczuk did in \cite[Section 2]{grytczuk1992ramanujan}.
 
 \begin{theo}\label{th:gen-dono}
  For positive integers $k, s$, and $n$, we have  
  \begin{align}\label{eq;gryt-gen}
  \sum\limits_{d|k}|c_d^{(s)}(n)| = 2^{\omega\left(\frac{k^s}{(k^s,n)_s}\right)}(k^s,n)_s.
  \end{align}

 \end{theo}

  \begin{proof}
   Note that $(a,b)_s$ is multiplicative in both $a$ and $b$ as a single variable function. Now if $(k,l)=1$, then 
   \begin{align*}
    \frac{k^sl^s}{(k^sl^s,n)_s} = \frac{k^s}{(k^s,n)_s} \frac{l^s}{(l^s,n)_s}
   \end{align*}
   so that \begin{align*}
            2^{\omega\left(\frac{k^sl^s}{(k^sl^s,n)_s}\right)} = 2^{\omega\left(\frac{k^s}{(k^s,n)_s} \frac{l^s}{(l^s,n)_s}\right)}.
           \end{align*}
Since $(k,l)=1$, so is $\frac{k^s}{(k^s,n)_s}$ and $\frac{l^s}{(l^s,n)_s}$ and so the above is equal to
\begin{align*}
  2^{\omega\left(\frac{k^s}{(k^s,n)_s}\right)+\omega\left( \frac{l^s}{(l^s,n)_s}\right)} =2^{\omega\left(\frac{k^s}{(k^s,n)_s}\right)}\, 2^{\omega\left( \frac{l^s}{(l^s,n)_s}\right)}
\end{align*}
so that the right hand side of equation (\ref{eq;gryt-gen}) is multiplicative. Now the left hand side of this equation is $h_n(k)$ in the notation used in the proof of Theorem \ref{th:gen-delange} and it multiplicative. Hence it enough to prove our claim  when $k$ is a prime power.

Let $k=p^r$ for some positive integer $r$. We have to prove that $h_n(p^r) = 2^{\omega\left(\frac{p^{rs}}{(p^{rs},n)_s}\right)}(p^{rs},n)_s$ for this $k$. We have already seen that if $p^s\nmid n$, then $h_n(p^r) = 2$. In this case, $(p^{rs},n)_s = 1$ and so $2^{\omega\left(\frac{p^{rs}}{(p^{rs},n)_s}\right)}(p^{rs},n)_s=2 = h_n(p^r)$. Now let $p^s|n$. If $r\leq e_p^{(s)}(n)$,  then
\begin{align*}
 h_n(p^r) &= p^{rs}
\end{align*}
and

\begin{align*}
 2^{\omega\left(\frac{p^{rs}}{(p^{rs},n)_s}\right)}(p^{rs},n)_s = 2^{\omega\left(\frac{p^{rs}}{p^{rs}}\right)}p^{rs} = p^{rs}
\end{align*} so that the claimed equality holds. On the other hand, if $r> e_p^{(s)}(n)$, then we may take $r-m=e_p^{(s)}(n)$ for some positive integer $m$. So $p^{(r-m)s}||_s n$. Then, as we have seen previously,
\begin{align*}
 h_n(p^r) &= 2p^{(r-m)s}.
\end{align*}
Now $(k^s, n)_s =  p^{(r-m)s}$ and $\frac{p^{rs}}{(p^{rs},n)_s} = p^{ms}$ so that $2^{\omega\left(\frac{p^{rs}}{(p^{rs},n)_s}\right)}(p^{rs},n)_s = 2p^{(r-m)s}$. This completes our proof.
\end{proof}

 It was observed previously in Remark \ref{rem:del-equality} that equality in Theorem \ref{th:gen-delange} holds when $n=m^s$ and $k$ is a multiple of $m\prod\limits_{p|m}p$. So the next statement follows easily.
 \begin{coro}
 When $n=m^s$, where $m$ is a positive integer, and $k$ is a multiple of $m\prod\limits_{p|m}p$, we have
  $n. 2^{\omega(k)}= 2^{\omega(\frac{k^s}{(k^s,n)_s})}(k^s,n)_s$.
 \end{coro}

 In \cite{grytczuk2002note}, A. Grytczuk used the sum $\sum\limits_{q|k}|c_q(n)|$ to evaluate $S(k,n):=\sum\limits_{d|k} 2^{\omega\left(\frac{k}{(k,n)}\right)} (k,n)$. We will now show how to evaluate an analogous sum.
\begin{theo}
 Let $S(k,n) = \sum\limits_{d|k} 2^{\omega \left(\frac{k^s}{(k^s,n)_s}\right)} (k^s,n)_s\, \mu(k/d)$, then we have
\begin{align}
S(k,n)=\begin{cases}
        \frac{J_s(k)}{J_s\left(\frac{k}{(k,n)}\right)}\quad&\text{if }\frac{k}{(k,n)}\text{ is square-free}\\
        0\quad&\text{ otherwise}
       \end{cases}
\end{align}
\begin{proof}
 From (\ref{eq;gryt-gen}), we have $\sum\limits_{d|k}|c_d^{(s)}(n)| = 2^{\omega\left(\frac{k^s}{(k^s,n)_s}\right)}(k^s,n)_s$. Then by the M\"{o}bius inversion formula, $|c_k^{(s)}(n)|=\sum\limits_{d|k} 2^{\omega\left(\frac{d^s}{(d^s,n)_s}\right)}(d^s,n)_s\, \mu(k/d)$. Using equation \ref{eq:crs-jordan}, we get

 \begin{align*}
  \left|\frac{J_s(n)\mu(\frac{k}{(k,n)})}{J_s(\frac{k}{(k,n)})}\right|=\sum\limits_{d|k} 2^{\omega\left(\frac{d^s}{(d^s,n)_s}\right)}(d^s,n)_s \mu(k/d).
 \end{align*}
Since $\mu$ is 0 on positive integers that are not square-free and $\pm 1$ on other positive integers, the statement in the theorem follows.

\end{proof}

\end{theo}

Now we give the converse of the absolute convergence in Theorem \ref{th:suff-condition} to conclude this paper.

A. Grytczuk \cite[Section 2]{grytczuk1992ramanujan} proved that the absolute convergence of the second series in Theorem \ref{th:suff-condition} (with $s=1)$  is sufficient for the absolute convergence of the first series. This is true for any $s$.

\begin{theo}
  If $\sum\limits_{q=1}^{\infty}  \sum\limits_{m=1}^{\infty}\frac{|f'(mq)|}{(mq)^s} |c_q^{(s)}(n^s)| <\infty$ then  $\sum\limits_{k=1}^{\infty} 2^{\omega(k)} \frac{|f'(k)|}{k^s} <\infty$.

\end{theo}

\begin{proof}
 Suppose $\sum\limits_{q=1}^{\infty}  \sum\limits_{m=1}^{\infty}\frac{|f'(mq)|}{(mq)^s} |c_q^{(s)}(n^s)| <\infty$.
 \begin{align*}
  \sum\limits_{q=1}^{\infty}  \sum\limits_{m=1}^{\infty}\frac{|f'(mq)|}{(mq)^s} |c_q^{(s)}(n^s)| &= \sum\limits_{k=1}^{\infty}  \sum\limits_{mq=k}\frac{|f'(mq)|}{(mq)^s} |c_q^{(s)}(n^s)|\\
  &= \sum\limits_{k=1}^{\infty}  \frac{|f'(k)|}{k^s} \sum\limits_{q|k} |c_q^{(s)}(n^s)|\\
  &= \sum\limits_{k=1}^{\infty}  \frac{|f'(k)|}{k^s} 2^{\omega(\frac{k^s}{(k^s,n)_s})}(k^s,n)_s.
 \end{align*}
 To arrive at the last step above, we used identity (\ref{eq;gryt-gen}).
Therefore $\sum\limits_{k=1}^{\infty}  \frac{|f'(k)|}{(k)^s} 2^{\omega\left(\frac{k^s}{(k^s,n)_s}\right)}(k^s,n)_s<\infty$. Since $\omega(m/n)\geq \omega(m)-\omega(n)$ with equality holding if and only if $(m,n)=1$, we have
\begin{align*}
 2^{\omega\left(\frac{k^s}{(k^s,n)_s}\right)} \geq \frac{2^{\omega(k^s)}} {2^{\omega\left((k^s,n)_s\right)}} \geq \frac{2^{\omega(k^s)}} {(k^s,n)_s}.
\end{align*}
Hence
\begin{align*}
 \sum\limits_{k=1}^{\infty}  \frac{|f'(k)|}{k^s} 2^{\omega\left(\frac{k^s}{(k^s,n)_s}\right)}(k^s,n)_s \geq \sum\limits_{k=1}^{\infty}  \frac{|f'(k)|}{k^s}2^{\omega(k^s)}.
 \end{align*}
 Since $\omega(k^s)=\omega(k)$, the theorem follows.
\end{proof}

\begin{rema}
 We would like to note here that K. R. Johnson evaluated the sum $\sum\limits_{q|k}|c^{(s)}_q(n)|$ (with $s=1$) in the other variable of the Ramanujan sum in \cite{johnson1983result}. Precisely, he evaluated $\sum\limits_{q|n}|c_q(n)|$. It is possible to evaluate such a sum in our case also.
\end{rema}


\end{document}